\documentclass[final, 11pt, nonatbib]{elsarticle}

\usepackage{amsmath, amssymb, amsthm}
\usepackage{tikz}
\usetikzlibrary{decorations.pathreplacing}
\usetikzlibrary{arrows.meta,positioning}
\usetikzlibrary{calc}
\usepackage{circuitikz}
\tikzset{
  freearrow/.style = {-{Stealth[length=3mm,width=2mm]}, blue, thick},
  clamparrow/.style = {-{Triangle[length=3mm,width=2mm]}, red, thick}
}

\tikzset{
  freearrow/.style = {-{Stealth[length=3mm,width=2mm]}, blue, thick},
  clamparrow/.style = {-{Triangle[length=3mm,width=2mm]}, red, thick}
}

\usepackage{mathrsfs}
\usepackage{bbm}
\usepackage{mathtools}
\usepackage{enumitem}
\usepackage{hyperref}
\usepackage{xcolor}
\hypersetup{
    colorlinks=true,
    linkcolor=blue,
    citecolor=blue,
    urlcolor=blue
}

\usepackage{caption}

\usepackage[normalem]{ulem}    

\usepackage{lineno}

\makeatletter
\let\c@author\relax
\makeatother
\usepackage[backend=biber, style=numeric-comp, sorting=none, doi=true, url=false, isbn=false]{biblatex}
\theoremstyle{definition}
\newtheorem{definition}{Definition}[section]

\theoremstyle{plain}
\newtheorem{theorem}[definition]{Theorem}

\newtheorem{lemma}[definition]{Lemma}
\newtheorem{example}[definition]{Example}
\newtheorem{remark}[definition]{Remark}

\numberwithin{equation}{section}

\begin{document}

\begin{frontmatter}

\title{A Conservation Law for Equilibrium Propagation and Coupled Learning}

\author[penn-math]{Joshua A.\ McGinnis\corref{cor1}}
\ead{jam887@sas.upenn.edu}
\author[penn-physics]{Adam G.\ Kline}
\author[penn-math,penn-bio]{Yoichiro Mori}

\cortext[cor1]{Corresponding author}

\affiliation[penn-math]{organization={Department of Mathematics, University of Pennsylvania}, city={Philadelphia}, state={PA}, country={USA}}
\affiliation[penn-physics]{organization={Department of Physics \& Astronomy, University of Pennsylvania}, city={Philadelphia}, state={PA}, country={USA}}
\affiliation[penn-bio]{organization={Department of Biology, University of Pennsylvania}, city={Philadelphia}, state={PA}, country={USA}}

\begin{abstract}
In this paper we show that the physical learning methods known as coupled learning (CL) and equilibrium propagation (EP) conserve a mass-like quantity in the trainable parameters in the continuous-time, small-nudging limit. We prove that this conservation holds in a broad range of physically relevant settings. We then show that the conservation law constrains the training dynamics in a way that makes convergence reliable in important settings for linear circuits. We conclude by discussing some practical implications of this conservation law. 
\end{abstract}


\end{frontmatter}


\section{Introduction}

Contrastive local learning techniques such as equilibrium propagation (EP)  \cite{scellier2017equilibrium} and coupled learning (CL)  \cite{stern2021learning} have recently drawn interest from numerous areas including biology, machine learning, and materials research. These methods are local since each adaptive part must update itself on the basis of signals in its immediate physical environment, and contrastive in that the relevant local signal is a difference between two system states. By imposing alternating constraints on a physical system and using CL or EP to adapt the interactions, the equilibrium or steady-state behavior is tuned to a desired point. This approach---locally tuning physical interactions in order to achieve desired equilibrium states---presents a difficult bilevel optimization problem~\cite{colson2007bilevel} which differs from common machine learning approaches and involves aspects of biological learning. The resulting loss is generally non-convex in the trainable parameters, and non-convex optimization is NP-hard in the worst case~\cite{murtyKabadi1987}; however, certain restricted topologies, such as the bipartite crossbar arrays we treat in Theorem~\ref{thm:crossbar}, admit more tractable structure.

Interest in physical learning from machine learning mainly stems from the idea that contrastive local learning networks (CLLNs) could drastically cut power and cost requirements \cite{christensen2022RoadmapNeuromorphic2022, dillavouMachineLearningProcessor2024}. The so-called inference phase, in which a given ML model is evaluated on input data, can be quite costly \cite{strubellEnergyPolicyConsiderations2019, luccioniPowerHungryProcessing2024}. By contrast, model evaluation is power-efficient in resistive circuits, since inference is achieved by allowing the circuit to relax into its physical steady-state under some constraints representing input data. In linear circuits, or circuits with linear resistors, this operating point minimizes the power dissipation. A closely related measure of dissipation is minimized in nonlinear circuits. In fact, power minimization is a common feature of physical learning approaches even beyond analog circuits \cite{vadlamaniPhysicsSuccessfullyImplements2020,stern2023learning}. Training is also much cheaper in principle, because the parameter update computations are local and carried out in-place. This should have a significant impact because moving data to and from memory is extremely costly \cite{strubellEnergyPolicyConsiderations2019}. Theoretical work has shown that CLLNs should be able to approach or even match typical performance on classic ML benchmarks like MNIST digit classification \cite{scellierUniversalApproximationTheorem2025}. Experiments have demonstrated learning in simpler tasks \cite{dillavou2022demonstration, dillavouMachineLearningProcessor2024}, but designs which succeed theoretically on ML benchmarks are typically too large or hard to train, and have yet to be built. There is therefore a need for theory to help develop powerful but experimentally tractable designs. 

CLLNs also capture several interesting aspects of learning in biology. This is in part by design, since learning in biological systems needs to be carried out using local rules: each neuron in a brain must adapt based on its environment. Such local learning should be contrasted with a standard implementation of gradient descent in a neural network model, in which each parameter's update depends on the states of all other parameters in the network. Clearly, each neuron in the brain or each cell in developing tissue cannot rely on information about every other cell in the system. While EP and CL address this locality, a given CLLN trained by either should not necessarily be considered neuromorphic by virtue of this fact, since it may not emulate any aspect of the brain, its parts, or its processes \cite{christensen2022RoadmapNeuromorphic2022, markovic2020physics}. More broadly, the CLLN approach is not generally concerned with realistically modeling any particular biological system. The generality of CLLNs allows for connections to a wide variety of learning phenomena in biology beyond brains, to proteins \cite{sternPhysicalNetworksBecome2024, sternPhysicalEffectsLearning2024} and even ecosystems \cite{dhanukaEcosystemsAdaptiveLiving2025}. For example, it has been shown \cite{guzmanMicroscopicImprintsLearned2024} that CLLNs can exhibit ``slow modes'' after training, corroborating arguments from evolutionary theory  \cite{husainPhysicalConstraintsEpistasis2020}. 

The bilevel optimization structure in CLLNs is important because it enforces low-power operation and leads to phenomena sometimes seen in biological learning, like slow modes.
However, this structure also poses the main theoretical challenge. 
The difficulty is mainly due to the fact that the steady-state is implicitly determined by the learning parameters and this dependence can be quite complicated. One must typically resort to simulating the physical system for each step of the learning algorithm, which can be inefficient, and it is typically not possible to make exact, explicit statements about the learning dynamics. In spite of this, we show that a physical symmetry of the system leads to an exact conservation condition satisfied by continuous EP and CL dynamics in the zero-nudge limit for a wide variety of circuits. This has several interesting theoretical consequences, and although this regime is difficult to achieve experimentally, it also has practical implications for CLLN design and training.

By considering linear circuits, we are able to use the conservation law to prove convergence in two classes of circuits. This includes a proof for crossbar arrays, which are the most expressive linear network for a given number of inputs and outputs (Thm~\ref{Thm:Crossbar_sufficient}). Furthermore, at the level of detail which we consider, symmetric clamping produces the same learning dynamics as zero-nudge clamping, hinting that dynamics with an approximately conserved quantity might be attainable in experiments. We also show that in batched updates, non-conserved nudging leads to drift in the overall scale of network parameters. This raises important practical considerations for implementation, since real adjustable edges have limits on their operating ranges \cite{dillavouMachineLearningProcessor2024}. If the scale of parameters is allowed to drift too far, pathologies and task failure can occur. On the other hand, there may be benefits to non-conserved nudging, such as model regularization.

\textbf{The outline of the paper is as follows}. We begin by reviewing the mathematical framework of EP  and CL in Section~\ref{Mathematical Set up}. In Section~\ref{C Laws} we prove the conservation laws in several different ways. We first give a proof that applies only to CL in Subsection~\ref{CL C Laws}. We then give a separate argument in Subsection~\ref{EP C Laws} that applies only to EP. This argument uses the fact that EP follows a gradient flow, and we explain this fact for completeness while also comparing it with the CL case. In Subsection~\ref{EP and CL C Laws} we present a single proof that works for both methods.

In Section~\ref{Cons of C} we explore the applications and implications of the conservation laws for the specific case of a linear circuit with trainable conductances or resistances. We also give physical intuition for the learning dynamics in the limit $\eta \to 0$ and note that this limit can be implemented using a local update rule without explicitly using the parameter $\eta$ in the case of linear circuits.

In Subsection~\ref{Convergece Argument} we show that the conservation law constrains the dynamics strongly enough that EP and CL must either converge to a solution or drive at least one trainable parameter to zero. We establish this result in the setting of a single input and a single output. We go on to prove exponential convergence for crossbar arrays. Finally, in \ref{sec:practical_implications}, we discuss the practical implication of conservation i.e.\@ how the choice of nudging method affects whether conductances vanish.

\section{Mathematical Set-Up}
\label{Mathematical Set up}

We begin with a general framework and introduce additional structure as necessary.
Let
\[
    G(x;k): \mathbb{R}^{N} \times \mathbb{R}^{M} \to \mathbb{R}
\]
be a smooth function, convex in $x$ for each admissible $k$. Here $x \in \mathbb{R}^N$ are the \textit{state variables} and $k \in \mathbb{R}^M$ are the \textit{trainable parameters}.

Let $v \in \mathbb{R}^I$ denote an input. Generally $I\ll N.$ We encode the input through a matrix $P \in \mathbb{R}^{N \times I}$, and we assume that in the \emph{free state}, the physical system minimizes its energy subject to enforcing this input:
\begin{equation}
    x_0
    = \arg\min_{x} G(x;k)
    \quad\text{subject to} \quad
    P^\top x = v.
    \label{eq:free-state}
\end{equation}
We assume that this constrained minimization problem has a unique solution for each admissible $k$; equivalently, $G(\cdot\,;k)$ is strictly convex on the affine subspace $\{x : P^\top x = v\}$.
During this stage, the parameters $k$ are held fixed.
The notation $x_0$ will become meaningful once we contrast it with the nudged/clamped states.

We encode the desired output through another matrix $Q \in \mathbb{R}^{N \times O}$ with $O \ll N$. We assume that the input variables and the output variables are disjoint. Given a target output $w \in \mathbb{R}^O$, we define the \emph{error}
\[
    r := Q^\top x_0 - w.
\]
The goal of physical learning is to adjust $k$ in order to decrease $r$.

Both equilibrium propagation and coupled learning modify the free state only at the output sites, using a small nudging parameter $\eta > 0$. The two methods differ in \emph{how} the nudge is applied.

\paragraph{Equilibrium Propagation (EP).}
EP introduces a small force proportional to the error:
\begin{equation}
    x_\eta 
    = \arg\min_{x} 
    \big(
        G(x;k) - \eta\, (Qr)\cdot x
    \big)
    \quad\text{subject to}\quad
    P^\top x = v.
    \label{eq:EP}
\end{equation}

\paragraph{Coupled Learning (CL).}
Coupled learning introduces the nudge as a constraint rather than as a force:
\begin{equation}
    x_\eta
    = \arg\min_{x} G(x;k)
    \quad\text{subject to}\quad
    P^\top x = v, 
    \qquad 
    Q^\top x =(1+\eta)Q^\top x_0 -\eta w .
    \label{eq:CL}
\end{equation}

In both cases, $x_\eta \to x_0$ as $\eta \to 0$.

\paragraph{Learning Dynamics.}

Let $k_i^{(n)}$ denote the $i$th parameter after the $n$th training iteration.
The physically implemented learning rule is of the form
\begin{equation}
    k_i^{(n+1)} 
    = k_i^{(n)}
      + \dfrac{\alpha}{\eta}\big(
      \partial_{k_i}G(x_\eta(k^{(n)});k^{(n)}) 
           - \partial_{k_i}G(x_0(k^{(n)});k^{(n)}) \big),
           \label{Learning_dynamics}
\end{equation}
where $\alpha>0$ is the learning rate.
\begin{remark}
    The notation $\partial_{k_i}G$ and $\partial_{k_i}\nabla_xG$ should always be interpreted as the derivative of the functions with respect to their $k_i$ argument. On the other hand, if other arguments of $G$ depend functionally on $k_i$ and we want the derivative to also act on these, we could write $\partial_{k_i}(G)$. For example
    \[
    \partial_{k_i}(G(x_0(k);k)) = \nabla_{x}G(x_0(k);k) \cdot \partial_{k_i}x_0(k) + \partial_{k_i}G(x_0(k);k).
\]
\end{remark}

We consider mathematically idealized continuous time training. Sending $\alpha \to 0$ and rescaling time gives the continuous-time limit
\begin{equation}
    \dot{k}_i
    =
    \dfrac{1}{\eta}\,\big(\partial_{k_i}
        G(x_\eta;k) - \partial_{k_i}G(x_0;k)
    \big).
    \label{eq:continuous-dynamics}
\end{equation}

In this paper, we are mainly concerned with the limit of small nudging $\eta \to 0$,
\begin{equation}
    \dot{k}_i
    =
   \partial_{k_i}
     \nabla_x G(x_0;k)\cdot \dfrac{d}{d\eta}x_\eta\bigg|_{\eta=0};
    \label{eq:leading-order-dynamics}
\end{equation}
Equation 
\eqref{eq:leading-order-dynamics} is the fundamental evolution equation for the remainder of the paper with an exception in Section \ref{sec:practical_implications}, where we also discuss what happens to conservation for non-zero $\alpha$ and $\eta$.

\section{Conservation Laws}
\label{C Laws}
We now prove the conservation law in three different ways. We prove it first for CL, then for EP, and finally, we give a proof that unifies both cases. We show these conservation laws for a general set of physically relevant circumstances characterized by the assumed identity, \eqref{eq:key-identity}, which holds for many, but not all, physical situations.

\subsection{Conservation for CL}
\label{CL C Laws}
Define the ``mass'' of the parameters to be
\begin{equation}  
    K := \frac12 \sum_{i=1}^M f_i(k_i).
    \label{eq:Conserved}
\end{equation}

\begin{remark}
    One could begin with a more general notion of a conserved quantity by taking 
    $K:\mathbb{R}^{M}\to\mathbb{R}$ arbitrary rather than the specific summation form used above.  
    In that case, the analogue of~\eqref{eq:key-identity} becomes a partial differential equation 
    that $K$ must satisfy.  At present we do not know the broadest conditions under which such a 
    conserved quantity exists; in general, the resulting PDE appears overdetermined and is unlikely 
    to admit nontrivial solutions.  
    \label{rmk:overdetermined}
\end{remark}
Differentiating in time gives
\[
    \dot{K}
    = \sum_{i=1}^M f_i'(k_i)\, \dot{k}_i
    = \sum_{i=1}^M f_i'(k_i)\,\partial_{k_i}
    \big( \nabla_x G(x_0;k)\big)\cdot \dfrac{d}{d\eta}x_\eta\bigg|_{\eta=0} .
\]
We {\bf assume} the following key identity holds:
\begin{equation}
    \sum_{i=1}^M f_i'(k_i)\, \partial_{k_i}\nabla_x G(x_0;k)
    = \nabla_x G(x_0;k).
    \label{eq:key-identity}
\end{equation}
Below, we characterize when the identity holds. Assuming \eqref{eq:key-identity}, we obtain
\[
    \dot{K}
    = \partial_\eta\big( G(x_\eta;k)\big)\bigg|_{\eta=0}.
\]
For CL, the nudged state $x_\eta$ satisfies an \emph{additional} constraint compared to $x_0$, and therefore
\[
    G(x_\eta;k) \ge G(x_0;k).
\]
We note that the inequality holds for positive and negative $\eta$. Thus,
\[
    \partial_\eta \big(G(x_\eta;k)\big)\bigg|_{\eta=0} = 0,
\]
from which we conclude the conservation law
\[
    \dot{K} = 0.
\]
We now give a lemma that characterizes when \eqref{eq:key-identity} holds. The lemma says that an energy with a certain separable structure gives rise to a conserved quantity.

\begin{lemma}
Suppose the energy separates as
\begin{equation}
    G(x;k) = \sum_{i=1}^{M} g_i(k_i)\,\widetilde{G}_i(x),
    \label{eq:Kirchoff_energy}
\end{equation}
where each $g_i$ is differentiable.  Assume further that the anti-derivative of 
$g_i(k)/g_i'(k)$ exists and denote it by $f_i(k)$.  Then
\[
    K(k) = \sum_{i=1}^{M} f_i(k_i)
\]
is a conserved quantity under the learning dynamics.
\label{Lemma:energy_conservation}
\end{lemma}
\begin{proof}
The proof follows by checking the identity~\eqref{eq:key-identity}. From the form of $G$, we have $\partial_{k_i} \nabla_x G(x_0; k) = g_i'(k_i)\,\nabla_x \widetilde{G}_i(x_0)$. Since $f'_i(k_i) = g_i(k_i)/g_i'(k_i)$, this gives, for each $i$,
\[
    f'_i(k_i)\,\partial_{k_i} \nabla_x G(x_0; k) = g_i(k_i)\,\nabla_x \widetilde{G}_i(x_0),
\]
which is the $i$-th summand of $\nabla_x G(x_0;k) = \sum_i g_i(k_i)\,\nabla_x \widetilde{G}_i(x_0)$. Summing over $i$ yields~\eqref{eq:key-identity}.
\end{proof}

Lemma~\ref{Lemma:energy_conservation} immediately yields the following two theorems, which cover two fundamental cases in physical learning for linear circuits.

\begin{theorem}[Conductance ``mass'' is conserved]
    Consider training the conductances in a linear resistive circuit.  
    In this case the energy takes the form
    \[
        G(x;\kappa) = \sum_{i=1}^{M} \kappa_i\,\widetilde{G}_i(x),
    \]
    where $\kappa_i$ represents the $i$th conductance.
    Then the conserved quantity is
    \[
        K(\kappa) = \frac{1}{2}\sum_{i=1}^{M} \kappa_i^2.
    \]
    \label{Thm:Cond}
\end{theorem}

\begin{theorem}[Resistance ``mass'' is conserved]
    Consider training the resistances instead of the conductances.  
    Then the energy is of the form
    \[
        G(x;\rho) = \sum_{i=1}^{M} \rho_i^{-1}\,\widetilde{G}_i(x),
    \]
    where $\rho_i$ denotes the $i$th resistance.
    In this case the conserved quantity is
    \[
        K(\rho) = \frac{1}{2}\sum_{i=1}^{M} \rho_i^2.
    \]
    \label{Thm:Res}
\end{theorem}

A conserved quantity exists when the circuit elements satisfy a scaling law. Specifically, let $\mathscr{I}_i(V,k_i)$ be the current through the $i$th element as a function of the voltage $V$ across it and the training parameter $k_i$. If changing $k_i$ only serves to amplify the current, i.e.,\@ has the form \ $\mathscr{I}_i(V,k) = g_i(k)\,\widetilde{\mathscr{I}}_i(V)$, then a conserved quantity exists.

\begin{example}{\textbf{Exponentially parameterized conductances.}}
Suppose all conductances are parametrized by $\kappa=e^{k}$, where $k$ is a trainable parameter. Then $\mathscr{I}(V,k)= e^{k}V$, and the conserved quantity is
\begin{equation*}
    K = \sum_{i=1}^{M}k_i
\end{equation*}
\label{Ex:exponential_parameter}
\end{example}

Different elements can have different current-voltage relationships and a conserved quantity can still exist. For example, on some edges, we can train resistances, on others we can train conductances, and still on others, train conductances with the parametrization given in Example \ref{Ex:exponential_parameter}. In such a case, there is a conserved quantity.

\begin{example}{\textbf{Diode in series with a conductance.}}
Suppose we have an ideal diode, which is off when the voltage is negative, in series with a trainable conductance, $\kappa$. The current is given by
\begin{equation*}
    \mathscr{I}(V,\kappa) = \begin{cases}
        \kappa V & V>0
        \\ 
        0 & V \leq 0.
    \end{cases}
\end{equation*}
As this has the separable form, a conservation law holds.
\end{example}

Conservation does not hold in all scenarios. Circuit elements that go untrained break conservation. Fixed current inputs break conservation. Circuit elements whose current-voltage relationships do not separate can break conservation.

\begin{example}{\textbf{Shockley diode.}}
    The current through such a diode with parameter $k$ may be modeled by
    \[\mathscr{I}(V,k)=e^{kV}-1.\]
    The intertwining of $k$ and $V$ in the exponential { \bf breaks} conservation.
\end{example}

\subsection{Conservation for EP}
\label{EP C Laws}

Conservation in EP can be seen as a consequence of the fact that EP performs
gradient flow on a scale invariant loss. To make this precise, assume the energy is
$p$–homogeneous in the parameters:
\[
    G(x;\lambda k) = \lambda^{p} G(x;k)
    \qquad \text{for all } \lambda>0.
\]
Under this assumption and using the strict convexity of $G(\cdot \, ;k)$ on the constraint surface, the error 
\[ r(k) := Q^\top x_0(k) - w \] 
is scale invariant ($0$–homogeneous) in $k$, and therefore the loss
\[
    \Phi(k) := \frac{1}{2}\|r(k)\|^{2}
\]
is scale invariant. As a consequence, $\Phi$ is constant along any ray
$\{\lambda k : \lambda>0\}$ emanating from the origin.  Since EP executes
gradient flow in $k$, its update direction is orthogonal to these level sets.
Therefore the squared radius
\[
    K(k) := \sum_{i=1}^{M} k_i^{2}
\]
must remain constant along the EP flow.  This yields both Theorems
\ref{Thm:Cond} and \ref{Thm:Res} for EP.

For completeness, we now explain precisely why EP performs gradient flow and
contrast this with what occurs in CL. The details of this computation are
used in the sequel for an independent proof of conservation for both EP and CL. 

\paragraph{Gradient of the Loss.}
By convexity of $G(\cdot \,;k)$, \eqref{eq:free-state} is equivalent to the KKT system
\begin{align}
\nabla_x G(x_0; k) + P\lambda = 0,\qquad P^\top x_0 = v, \label{eq:kkt}
\end{align}
for a Lagrange multiplier $\lambda\in\mathbb{R}^I$.

Taking the total derivative of \eqref{eq:kkt} with respect to $k_i$ gives
\begin{align}
\begin{bmatrix}
H & P\\
P^\top & 0_{I \times I}
\end{bmatrix}
\begin{bmatrix}
\partial_{k_i}x_0\\[2pt]
\partial_{k_i}\lambda
\end{bmatrix}
= -
\begin{bmatrix}
\partial_{k_i}\nabla_x G\\
0_{I \times 1}
\end{bmatrix},\qquad
H := \nabla_x^2 G(x_0;k).\label{eq:lin-kkt},
\end{align}
 where $0_{N \times M}$ and $N \times M$ matrix of zeros. Let
\begin{align}
S := \begin{bmatrix} Id_{N\times N}\\ 0_{I\times N}\end{bmatrix},\qquad
R := S^\top
\begin{bmatrix}
H & P\\
P^\top & 0_{I \times I}
\end{bmatrix}^{-1}\!
S,\label{eq:R}
\end{align}
where $Id_{N \times N}$ is the identity matrix. By the strict convexity of $G(\cdot\,;k)$ on the constraint surface, $H$ is positive definite on $\ker P^\top$, which ensures the bordered Hessian is invertible. Hence
\begin{align}
\partial_{k_i}x_0 = - R\,\partial_{k_i}\nabla_x G.\label{eq:phi-sens}
\end{align}
and
\begin{equation}\partial_{k_i} \Phi(k) = \langle Qr,  \partial_{k_i}x_0(k)\rangle =-\langle RQr,  \partial_{k_i}\nabla_x G\rangle . 
\label{eq:dynamics-grad}
\end{equation}

\paragraph{EP Executes Gradient Flow.}

The KKT conditions for the EP minimization problem \eqref{eq:EP} read
\begin{align}
\nabla_x G(x_\eta) + P\lambda_\eta = \eta Qr,\qquad P^\top x_\eta = v.\label{eq:kkt-eta-EP}
\end{align}
Differentiating in $\eta$ at $\eta=0$ yields
\begin{align}
\begin{bmatrix}
H & P\\
P^\top & 0
\end{bmatrix}
\begin{bmatrix}
y\\
\mu
\end{bmatrix}
=
\begin{bmatrix}
Qr\\
0
\end{bmatrix}, \quad \mu :=\left.\dfrac{d}{d\eta}\lambda_\eta\right|_{\eta=0} ,\quad
y := \left. \frac{d}{d\eta}x_\eta\right  |_{\eta=0}\label{eq:psi}
\end{align}
We can solve $y = RQr.$ Hence
\begin{align}
\dot{k}_i=\lim_{\eta\to 0}\frac{1}{\eta}\Big(\partial_{k_i}G(x_\eta;k)-\partial_{k_i}G(x_0;k)\Big)
= \big\langle y, \nabla_x \partial_{k_i}G\big\rangle
= \big\langle R Qr,\partial_{k_i}\nabla_x G\big\rangle\label{eq:dynamics-EP},
\end{align}
which is the negative of $\partial_{k_i}\Phi$.
The dynamics of the loss function are given by
\begin{equation}
    \dot{\Phi} = - \sum_{i=1}^M\big\langle R Q r,\partial_{k_i}\nabla_x G\big\rangle^2=- \sum_{i=1}^M\dot{k}_i^2.
\end{equation}
Therefore the loss must decrease when the $k$ are not at a fixed point.

\paragraph{Not Quite Gradient Flow for CL.}
Recall the nudged energy minimization problem \eqref{eq:CL} with minimizer $x_{\eta}$. The KKT conditions read
\begin{align}
\nabla_x G(x_\eta;k) + P\lambda_\eta +Q\nu_\eta = 0,\quad P^\top x_\eta = v, \quad Q^\top x_\eta = Q^\top x_0 + \eta r
\label{eq:kkt-eta-CL}
\end{align}
with Lagrange multipliers given by $ \lambda_\eta \in \mathbb{R}^I$ and $\nu_\eta \in \mathbb{R}^{O}.$ Differentiating in $\eta$ at $\eta=0$ yields
\begin{align}
\begin{bmatrix}
H & P & Q\\
P^\top & 0 & 0 \\
Q^\top & 0 & 0
\end{bmatrix}
\begin{bmatrix}
y\\[2pt]
\mu\\
\omega
\end{bmatrix}
=
\begin{bmatrix}
0\\
0\\
r
\end{bmatrix}, \quad  \omega =\left.\dfrac{d}{d\eta}\nu_\eta\right|_{\eta=0} ,\quad \mu =\left.\dfrac{d}{d\eta}\lambda_\eta\right|_{\eta=0} ,\quad
y:= \left. \frac{d}{d\eta}x_\eta\right |_{\eta=0}  \label{eq:psi2}
\end{align}
Solving, we find $y = R Q (Q^\top RQ)^{-1}r$. Hence
\begin{align}
\dot{k}_i=\lim_{\eta\to 0}\frac{1}{\eta}\Big(\partial_{k_i}G(x_\eta;k)-\partial_{k_i}G(x_0;k)\Big)
= \big\langle y, \nabla_x \partial_{k_i}G\big\rangle
= \big\langle R Q u_r,\partial_{k_i}\nabla_x G\big\rangle.\label{eq:dynamics-CL}
\end{align}
where $u_r:=(Q^\top RQ)^{-1}r$, which is very similar to the dynamics of EP. 

In the special case where $r$ is scalar so $Q$ is column vector, $(Q^\top RQ)^{-1}$ is a positive scalar quantity. The dynamics of the loss function are given by
\begin{equation}
\dot{\Phi}
= - (Q^\top RQ)^{-1}\sum_{i=1}^M\big\langle R Q r,\partial_{k_i}\nabla_x G\big\rangle^2
= - (Q^\top RQ)\sum_{i=1}^M \dot{k}_i^2
\;\le\;0.
\end{equation}
Therefore the loss must decrease whenever $\dot{k}\neq 0$. This rules out limit cycles for the special case of a one dimensional output.

\subsection{Conservation for EP and CL}
\label{EP and CL C Laws}

Now we can use the form of the KKT equations to prove the conservation law for both EP and CL. As this proof works for both EP and CL, it is in some sense the essential proof. 

Suppose again that we have defined $K$ as in \eqref{eq:Conserved}, and we have that \eqref{eq:key-identity} holds.  Then we only need to check (for both EP and CL) that

\begin{equation}
   \dot{K}=\langle RQr, \nabla_x G(x_0 ;k)\rangle =\langle RQu_r, \nabla_x G(x_0 ;k)\rangle =0
\end{equation}

However, using \eqref{eq:kkt}, this is equivalent to checking that 

\begin{equation}
    \langle P^\top RQr, \lambda \rangle =\langle P^\top RQu_r, \lambda\rangle =\langle P^\top y,\lambda\rangle=0,
\end{equation}
which holds by considering \eqref{eq:psi} and \eqref{eq:psi2}.

Therefore, we have obtained Lemma \ref{Lemma:energy_conservation} for both EP and CL and as a consequence, Theorems \ref{Thm:Cond}
 and \ref{Thm:Res}.

\section{Applications of Conservation}
\label{Cons of C}
In this section we restrict our attention to linear circuits. We use the conserved quantities identified above to     
establish convergence results for both CL and EP. Specifically, in Subsection \ref{Convergece Argument}, we show that the conservation law constrains the      
learning dynamics strongly enough to force convergence to a solution, provided the circuit does not degenerate by
losing all paths between input and output. We prove this in full generality for bipartite circuits (crossbar arrays),
where we obtain exponential convergence of the loss. For general circuits with a single input and output, we prove
convergence under the assumption that no conductance reaches zero. In Subsection \ref{sec:practical_implications}, we discuss the practical implication of conservation i.e.\@ how the choice of nudging method affects whether conductances vanish.

\paragraph{Linear Circuits.}
Here we lay out the notation for a linear circuit. Let $\mathcal{V}=\{1,\ldots,N\}$ be the set of nodes. The state vector $x = (x_1,\ldots,x_N) \in \mathbb{R}^N$ gives voltages at nodes. We define the edge set
\begin{equation}
    \mathcal{E} = \{(i,j) \ | \ i,j\in[N], i \sim j, i >j\},
\end{equation}
where $[N]:=\{1,\ldots,N\}$, and we have $|\mathcal{E}| = M$. We may lexicographically order the edges and use this ordering to define the $M$-tuple of conductances $\kappa$. Then $e$ can be construed as the edge or the index of that edge, where $e \in \{1,2,\ldots,M\}$.

For $e \in \mathcal{E}$ with $e = (i,j)$, we define the incidence vector as $d_e:=e_{i} -e_{j}$ where $e_i,e_j \in  \mathbb{R}^N$ are the natural unit basis vectors. Analogously $\kappa_e = \kappa_{i,j}$.

We distinguish explicitly between conductances $\kappa$ and resistances $\rho$. The \textit{circuit Laplacian} $L$ is defined as
\begin{equation}
    L :=\sum_{e=1}^{M}\kappa_ed_ed_e^\top = \sum_{e=1}^{M}\dfrac{1}{\rho_e}d_ed_e^\top.
    \label{The_Laplacian}
\end{equation}
Assuming all conductances are positive and the circuit is connected, $L$ is symmetric and positive semi-definite with $\ker L = \operatorname{span}\{\mathbf{1}\}$. Here $\mathbf{1} \in \mathbb{R}^{N}$ is the vector of all ones.

The energy function $G$ is given by
\[ G(x ; \kappa) := \dfrac{1}{2}\langle x, Lx\rangle,\]
which is convex in $x$. We assume that the input matrix $P$ is chosen so that the constrained minimization of $G$ subject to $P^\top x = v$ has a unique solution; equivalently, $\ker L \cap \ker P^\top = \{0\}$. To handle the input constraints, we define the \textit{input-constrained Laplacian} as
\begin{equation}
    \hat{L} = \begin{bmatrix}
        L & P \\ P^\top & 0
    \end{bmatrix}.
    \label{eq:hat-L}
\end{equation}
Since $L$ is positive definite on $\ker P^\top$, $\hat{L}$ is invertible. We define  $$\hat{x} := \begin{bmatrix}x \\ \lambda\end{bmatrix}, \quad  \hat{v} := \begin{bmatrix}0_{N \times 1} \\ v\end{bmatrix}, \quad \hat{d}_e := \begin{bmatrix}d_e \\ 0_{I \times 1}\end{bmatrix}, \quad \text{and} \quad \hat{Q}:= \begin{bmatrix} Q \\ 0_{I \times O}\end{bmatrix},$$ where $\lambda$ are the Lagrange multipliers for the input constraints, Kirchhoff's law gives $\hat{L}\hat{x}_0 = \hat{v}$. The multiplier $P\lambda$ may be interpreted as the current injection needed to sustain the input nodes at the constrained voltages $v$. 

Our conservation theorems, Theorem \ref{Thm:Cond} and \ref{Thm:Res}, tell us that 
$
    \frac{1}{2}\sum_{e=1}^{M} \kappa_e^2
$
is conserved when we train on conductances, while
$
    \frac{1}{2}\sum_{e=1}^{M} \rho_e^2
$
is conserved when we train on resistances. As a consequence, it ensures that no conductances can grow without bound and thus eliminates the need for an artificially enforced upper boundary. However, the worry for negative conductances persists, and we make an extended remark on this below. 
\paragraph{Symmetric Clamping.}
An alternative to the learning rule \eqref{eq:continuous-dynamics} is the symmetric or central-difference update
\begin{equation}
    \dot{\kappa}_e
    = \dfrac{1}{2\eta}\Big(\partial_{\kappa_e}G(x_\eta;\kappa) - \partial_{\kappa_e}G(x_{-\eta};\kappa)\Big).
    \label{eq:symmetric-clamping}
\end{equation}
In linear circuits, $x_\eta$ depends linearly on $\eta$: we can write $x_\eta = x_0 + \eta y$ where $y = \left.\frac{d}{d\eta}x_\eta\right|_{\eta=0}$. Since $G$ is quadratic in $x$, the quantity $\partial_{\kappa_e}G(x_\eta;\kappa)$ is quadratic in $\eta$. The central difference \eqref{eq:symmetric-clamping} eliminates the $\eta$ term exactly, and we obtain
\[
    \dot{\kappa}_e = \partial_{\kappa_e}\nabla_x G(x_0;\kappa) \cdot y,
\]
which is precisely the $\eta \to 0$ dynamics \eqref{eq:leading-order-dynamics}. Therefore, in linear circuits, symmetric clamping at any finite $\eta$ conserves $K$ exactly. For non-linear circuits, we expect conservation to hold up to $O(\eta^2)$ under symmetric clamping.

\paragraph{Physical Interpretation.}
Attempting to implement EP or CL physically with a very small nudging parameter
$\eta$ may be difficult, since it requires reliably imposing a tiny perturbation
in hardware. By the preceding observation, symmetric clamping at any finite $\eta$ achieves the same dynamics and therefore conserves $K$ for all $\eta$. We can also note that these dynamics can be implemented directly without any nudging parameter at all.

Applying \eqref{eq:dynamics-EP} to the conductance circuit, the EP update in the limit $\eta\to 0$ is
\begin{equation}
    \dot{\kappa}_e
      = \bigl(d_e^\top RQ\, r\bigr)\,
      \bigl(d_e^\top R v\bigr),
      \label{eq:EP-cirucit}
\end{equation}
while from \eqref{eq:dynamics-CL} the CL update is
\begin{equation}
    \dot{\kappa}_e
    = \bigl(d_e^\top RQ\, u_r\bigr)\,
      \bigl(d_e^\top R v\bigr),
      \label{eq:CL-circuit}
\end{equation}
where $u_r := (Q^\top RQ)^{-1} r$ and $R = S^\top \hat{L}^{-1} S$ as defined in \eqref{eq:R}.

The quantity $d_e^\top R v = \hat{d}_e^\top \hat{L}^{-1} \hat{v}$ is the voltage drop across edge $e$ in
the free circuit (i.e.\ the circuit with no nudging). The quantities
$d_e^\top RQ\,r$ and $d_e^\top RQ\,u_r$ are voltage drops in a
circuit with current inputs clamped at output nodes, \emph{but} with the additional feature that the input nodes are clamped to $0$ voltage. This can be seen by considering \eqref{eq:psi} for EP and \eqref{eq:psi2} for CL. In the EP case, the current input at the output nodes is exactly the same as taking $\eta =1$ in the clamping step. For CL, the current input at the output node is exactly the current needed to clamp the output node voltage to the error $r$, since $(Q^{\top} RQ)^{-1}$ is a discrete version of a Dirichlet-Neumann operator. We refer to these as \emph{no-nudge-clamping} and their physical interpretation is important to us in the sequel.

Thus one can perform learning without introducing a small
$\eta$: instead of taking the difference between the clamped and free voltage
drops, one simply measures both voltage drops and multiplies them to obtain the update in~\eqref{eq:EP-cirucit} or
\eqref{eq:CL-circuit}. Taking a difference of voltages becomes important in the non-linear setting.

\paragraph{Handling Vanishing Conductances.} One potential difficulty in the learning process is that one or more
conductances may converge to zero. When this occurs, we treat the corresponding edge as having been removed from the circuit to avoid negative conductances. We note that removing edges in such a fashion maintains the conserved quantity. Still, the network risks becoming
disconnected.  This situation is subtle: if the input and output nodes remain
connected, nothing essential is lost, and we can restrict our attention to the
remaining connected subnetwork.  However, it is a priori difficult to rule out
the possibility that training causes all the connections between the inputs and outputs to become severed.

Another possible rule for the learning dynamics, which is common among practitioners of physical learning and would not involve removing edges with zero conductance, is the following. Rather than deleting an edge whose conductance has vanished, we let the conductance update only when the update would have it increase from $0$. We call this the \textit{reflecting rule}, and define its dynamics by:
\[ \dot{\kappa}^{\rm (ref)}_{e} \coloneqq \begin{cases}
\dot{\kappa}_e & \kappa_e>0 \ \text{or} \ \dot{\kappa}_e\geq 0
\\
0 & \text{Otherwise}
\end{cases}.\]
The rule requires that $\dot{\kappa}_e$ is actually defined for an edge $e$ whose conductance is $0.$ Unfortunately, it cannot yet be disproven that pathological situations may lead to isolated nodes with ill-defined voltages, thus making the reflecting rule ill-defined. However, if such a rule can be defined, it preserves the conservation law.

\subsection{Convergence of CL and EP}
\label{Convergece Argument}
Throughout the remainder of the paper ``clamping", ``clamped state", etc.\ refer to no-nudge-clamping defined in reference to the first terms of the right-hand side of  \eqref{eq:EP-cirucit} and \eqref{eq:CL-circuit}. We also adopt the following convention:
\[
\begin{aligned}
    &\text{\textbf{Whenever an edge's conductance reaches $0$ during training,}}
    \\
    &\text{\textbf{we remove that edge from the circuit.}}
\end{aligned}
\]
This procedure is compatible with the conservation law, since an edge with
zero conductance contributes nothing to the conserved quantity at the moment it
is removed.  \textit{In particular, the circuit cannot vanish completely.}  This is
highly suggestive that the algorithm cannot fail to converge to a fixed point,
since it seems implausible that the circuit could degenerate in such a way that
the input and output nodes become disconnected. Making this intuition rigorous,
however, is difficult. 

For now, we obtain an actual theorem by simply excluding the case in which
conductances reach $0$.  We consider a circuit with a single input node, a
single output node, and a ground.  Without loss of generality, we label the input
node $1$, the output node $2$, and, following our convention, the ground node
$0$.  We impose the constraint $x_1 = 1$ at the input node.  We assume that, at
the outset, none of the conductances are zero and that the input, output, and
ground nodes are all connected.  The target is
$x_2 = w$ with $w\in(0,1)$, so the desired output is clearly feasible.  A
representative example is shown in Figure~\ref{fig:Basic Circuit}.
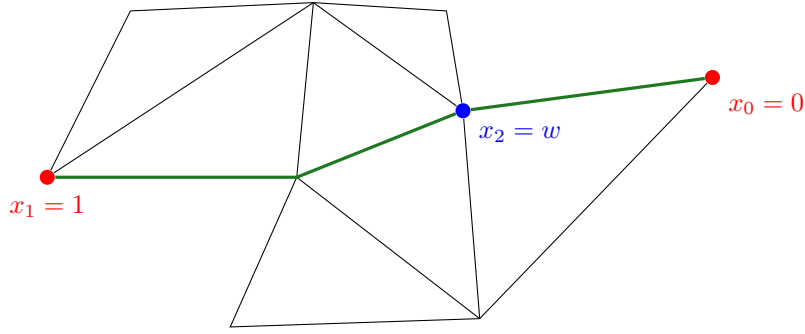
\begin{figure}[ht]
\centering
\begin{tikzpicture}[scale=1.1, every node/.style={font=\small}]
\definecolor{mygreen}{RGB}{30,120,30}

\node[fill=red,circle,inner sep=2pt,
      label={[red]below:$x_1=1$}] (A) at (0,0) {};

\coordinate (J) at (3,0);

\node[fill=blue,circle,inner sep=2pt,
      label={[blue]below right:$x_2 = w$}] (C) at (5,0.8) {};

\node[fill=red,circle,inner sep=2pt,
      label={[red]below right:$x_0 = 0$}] (B) at (8,1.2) {};

\coordinate (U1) at (1,2);
\coordinate (U2) at (3.2,2.1);
\coordinate (U3) at (4.8,2.0);

\coordinate (D1) at (2.2,-1.8);
\coordinate (D2) at (5.2,-1.7);

\draw[thin] (A) -- (U1) -- (U2) -- (J);
\draw[thin] (A) -- (U2);

\draw[thin] (U2) -- (U3) -- (C) -- (J) -- cycle;
\draw[thin] (U2) -- (C);

\draw[thin] (J) -- (D1) -- (D2) -- (C) -- cycle;
\draw[thin] (J) -- (D2);

\draw[thin] (C) -- (B);

\draw[thin] (B) -- (D2);

\draw[very thick, mygreen] (A) -- (J) -- (C) -- (B);

\end{tikzpicture}
\caption{A basic circuit with input $x_1 = 1$, desired output $x_2 = w$, ground
at $x_0 = 0$, and a path connecting all three.}
\label{fig:Basic Circuit}
\end{figure}

We have the following theorem.

\begin{theorem}
    In the single-input/single-output setting described above, for both EP and
    CL, either the circuit converges to a configuration with $x_2 = w$, or at
    least one conductance tends to $0$.
    \label{Lemma:converge}
\end{theorem}

\begin{proof}
Since the conductance vector $\kappa$ evolves on the intersection of the sphere
$\|\kappa\| = C$ (by conservation) with the closed positive orthant
$\{\kappa_i \ge 0\}$, the state space is compact.  Thus every trajectory has at
least one accumulation point.  Any such limit point either lies on the boundary
(where at least one conductance is zero) or is an interior point with all
conductances strictly positive.  In the latter case, the limit point must be a
fixed point of the dynamics.

It therefore suffices to show that any fixed point with all conductances
strictly positive satisfies $x_2 = w$.  Fixed points require
$\dot{\kappa}_e = 0$ for all $e$.  Recalling the EP update \eqref{eq:EP-cirucit}, for a fixed point we have that for all $e$
\begin{equation}
 0 = \bigl(d_e^\top RQ r\bigr)\,
                     \bigl(d_e^\top R v\bigr),
\end{equation}
while using \eqref{eq:CL-circuit} yields
\begin{equation}
    0 = \bigl(d_e^\top RQ u_r\bigr)\,
                     \bigl(d_e^\top R v\bigr)
\end{equation}
for CL.

Choose $e$ so that it lies on a path from node $2$ (the output) to
node $0$ (the ground) and is connected directly to ground at one end. Since node $1$ is connected to node $2$ which is grounded via $e$, the maximum principle applies: the non-ground node connected to $e$ cannot obtain a minimum $(0)$ or maximum $(1)$ in the free state. Thus in the free state, there must be voltage drop over $e$:
$d_e^\top R v \neq 0$. Consequently, a fixed point with all
conductances positive must satisfy
\[
    d_e^\top RQ r = 0 \quad\text{for EP}, \qquad
    d_e^\top RQ u_r = 0 \quad\text{for CL}.
\]
For EP (CL), in this clamped state, no interior node along paths connecting node $2$ to node $0$ may obtain a voltage of $0$ due to the maximum principle, as long as $u_r \neq 0$ ($r \neq 0$). Therefore, there must be voltage drop across $e$. We note that $u_r =0$ if and only if $r = 0$ since the inputs and outputs are connected, and therefore, the only fixed points for which conductances do not vanish satisfy $r =0.$

\end{proof}

\begin{remark}[On the possibility of edge death]\label{rmk:edge_death}
Theorem~4.1 
leaves open the possibility that one or more conductances tend to zero,
potentially disconnecting the input from the output.  We argue briefly that
this is unlikely in practice, and can be ruled out in certain cases.

Consider a circuit with a single input node, a single output node, and a
ground, and suppose there is exactly one ``bridge'' edge whose removal would
disconnect the input from the output.  Let $\kappa_1$ denote its conductance
and suppose $\kappa_1$ is small while all other conductances remain order one.
Then in the free state, the voltage drop across the bridge is order one
(the two components are nearly decoupled, so their voltages are set by their
respective boundary conditions).  Meanwhile, in the clamped state, the
error $r \neq 0$ forces a nontrivial voltage drop across the bridge as well,
by the maximum principle.  Since the update rule multiplies these two voltage
drops, we obtain $\dot{\kappa}_1 > 0$: the bridge edge is pushed away from zero.

This argument relies on the specific topology (a single bridge) and does not
easily extend to the case where multiple edges vanish simultaneously, since
the voltage structure can degenerate in more complicated ways.
In the next subsection, we consider bipartite circuits, where the
conservation law is strong enough to give a complete convergence result
without needing to track individual edges.
\end{remark}

\subsubsection{Convergence in bipartite circuits}

The crossbar array is a bipartite circuit in which inputs form one partition and outputs form another, and all input-output pairs have an edge. Two representations of such a circuit are found in Figure \ref{fig:bipart}. Circuits of this type are widely studied in neuromorphic computing, where the edges used are frequently not linear resistors \cite{xiaMemristiveCrossbarArrays2019}. Moreover, some crossbars are used as voltage-in current-out devices, but here we consider the case in which the outputs are voltages. 

\begin{figure}
    \centering
    \begin{tikzpicture}[scale=1.0, every node/.style={font=\small}]

\begin{scope}[xshift=0cm]

\foreach \i in {1,...,4} {
  \node[fill=red, circle, inner sep=2.5pt, label=left:{$v_{\i}$}] (I\i) at (0, -\i*1.2 + 0.6) {};
}

\foreach \j in {1,2} {
  \node[fill=blue, circle, inner sep=2.5pt, label=right:{$w_{\j}$}] (O\j) at (3.5, -\j*1.8 + 0.3) {};
}

\foreach \i in {1,...,4} {
  \foreach \j in {1,2} {
    \draw[thin, black!70] (I\i) -- (O\j);
  }
}

\node at (1.75, -5.0) {(a) Bipartite graph};

\end{scope}

\begin{scope}[xshift=8cm]

\foreach \i in {1,...,4} {
  \draw (-0.5, -\i*1.0) -- (3.5, -\i*1.0);
  \node[left] at (-0.5, -\i*1.0) {$v_{\i}$};
}

\foreach \j in {1,2} {
  \draw (\j*1.5, 0) -- (\j*1.5, -5.0);
  \node[above] at (\j*1.5, 0) {$w_{\j}$};
}

\foreach \i in {1,...,4} {
  \foreach \j in {1,2} {
    \filldraw[fill=white, draw=black] (\j*1.5 - 0.2, -\i*1.0 - 0.15) rectangle (\j*1.5 + 0.2, -\i*1.0 + 0.15);
    \node[font=\tiny, above right] at (\j*1.5 + 0.2, -\i*1.0) {$\kappa_{\i\j}$};
  }
}

\node at (1.75, -5.7) {(b) Crossbar array};

\end{scope}

\end{tikzpicture}
    \caption{A bipartite graph circuit on the left together with its corresponding crossbar array representation}
    \label{fig:bipart}
\end{figure}
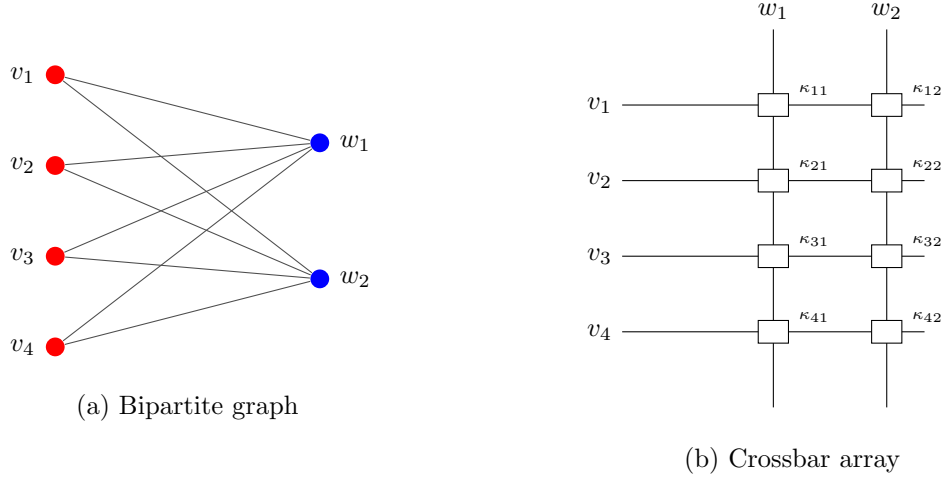

Among all design choices for linear circuits, the crossbar array is  sufficient, when it comes to specifying voltage input-output relationships. We make this statement precise with Theorem \ref{Thm:Crossbar_sufficient} below. Given a linear circuit $C$, with edge conductances $\kappa_e$, define $A_C(\kappa):\mathbb{R}^I\to \mathbb{R}^O$ as the input-output map. Specifically, for an input voltage constraint $P^\top x_0 = v,$ we have that $ A_C(\kappa)v:= Q^\top x_0$. The theorem is as follows:

\begin{theorem}
[Crossbar is sufficient]
    \label{Thm:Crossbar_sufficient}
    For any linear circuit $C$ with conductances $\kappa_e\geq0$ and input-output map $A_C(\kappa)$, there exists a crossbar array $C'$ with conductances $\kappa'_{ij}$ which has the same input-output map. That is, $A_{C'}(\kappa') = A_C(\kappa)$.
    \label{Thm:Crossbar}
\end{theorem}

\begin{proof}
    Let $\kappa'_{ij}$ be the conductance of the resistor connecting input node $j$ to output node $i$. There is one such resistor for each input-output pair by assumption, and no others. Then the entries of $A_{C'}(\kappa')$ are given by
    \[
    [A_{C'}(\kappa')]_{ij} =  \frac{\kappa'_{ij}} {\sum_k \kappa'_{ik}}
    \]
  Crossbars fully parameterize the set of $I \times O$ right stochastic matrices, defined by
    \[
    \mathcal{S} := \{A:\mathbb{R}^I\to \mathbb{R}^O|[A]_{ij} \geq 0, A\mathbf{1}_I = \mathbf{1}_O\},
    \]
    where $\mathbf{1}_K\in \mathbb{R}^{K}$ is the vector of $K$ ones. That is, for any $A \in \mathcal{S}$, there exists $\kappa'_{ij}\geq0$ such that $A = A_{C'}(\kappa')$. 

     Now, showing that $A_C(\kappa) \in \mathcal{S}$ completes the proof. First, $A_C(\kappa)\mathbf{1}_I =\mathbf{1}_O$, because for the input constraint $P^\top x_0 = \mathbf{1}_I$, the configuration $x_0 = \mathbf{1}_N$ is the global minimum ($G(x_0;\kappa) = 0$). Second, each entry $a_{ij} = e_i^\top A_C(\kappa) e_j \geq 0$ by the maximum principle, since acting on the right with $e_j$ is represents imposing a voltage constraint which is zero on all input nodes except node $j$, where it the 
    voltage is one.
    
\end{proof}

We now consider the learning dynamics of a crossbar array. Since the outputs are connected only to inputs with fixed voltages, each output can be treated independently. The problem therefore reduces to a single output node connected to $N$ inputs by conductances $\kappa_1, \ldots, \kappa_N > 0$.  Without loss of generality we order the input voltages so that $x_1 \geq x_2 \geq \cdots \geq x_N$ with $x_1 \neq x_N$.
 
Let
\[
\bar{\kappa} = \sum_{i=1}^{N} \kappa_i
\]
and define the output voltage in the free state as
\[
\mu = \bar{\kappa}^{-1} \sum_{i=1}^{N} \kappa_i\, x_i.
\]
The energy of the crossbar is
$G(x_{\mathrm{out}}; \kappa) = \frac{1}{2}\sum_{i} \kappa_i(x_{\mathrm{out}} - x_i)^2$,
and one checks that $\mu$ is its minimizer.
Given a target $w \in (x_N, x_1)$, the error and loss are
\[
r = \mu - w, \qquad \Phi = \dfrac{1}{2}\,r^2.
\]
\begin{theorem}\label{thm:crossbar}                   
  In the setting described above, for both CL and EP, the loss                 
  \[\Phi(t) = \dfrac{1}{2}\big(\mu(t) - w\big)^2 \]
  converges to zero exponentially. This holds even if some conductances reach
  zero and are removed during training.
  \end{theorem}

  \begin{proof}

  We give the proof for CL; the EP case is addressed in Remark~\ref{Rmk:EP_Crossbar} below. Computing the CL update gives
\[
\dot{\kappa}_i
= \lim_{\eta \to 0}
  \dfrac{1}{\eta}\Big(\partial_{\kappa_i}G\big(\mu + \eta r;\, \kappa\big)
  - \partial_{\kappa_i}G\big(\mu;\, \kappa\big)\Big)
= -r\,(x_i - \mu).
\]

Write $\mathcal{A}(t) := \{i : \kappa_i(t) > 0\}$ for the set of
\emph{active} edges.  Since
$\partial_{\kappa_i}\mu = \bar{\kappa}^{-1}(x_i - \mu)$, the loss dynamics
restricted to active edges are
\[
\dot{\mu}
= \bar{\kappa}^{-1}\sum_{i \in \mathcal{A}} \dot{\kappa}_i\,(x_i - \mu)
= -\dfrac{r}{\bar{\kappa}} \sum_{i \in \mathcal{A}} (x_i - \mu)^2.
\]
Writing $V_{\mathcal{A}} := \sum_{i \in \mathcal{A}} (x_i - \mu)^2$, we obtain
\begin{equation}\label{eq:crossbar_loss}
\dot{\Phi} = r\,\dot{\mu} = -\dfrac{2\Phi}{\bar{\kappa}}\, V_{\mathcal{A}}
\leq 0.
\end{equation}
We need to show that $V_{\mathcal{A}}$ remains bounded below.

Whenever an edge's conductance reaches zero, we remove it from the circuit.
Since $\kappa_i = 0$ at the moment of removal, neither $\mu$ nor $\bar{\kappa}$
changes discontinuously, and the removed edge contributes nothing to the
conserved quantity $K = \frac{1}{2}\sum \kappa_i^2$.
The ODE continues on the reduced system.

If $r = 0$ then $\Phi = 0$ and the target is already achieved.
Otherwise, since $\dot{\Phi} \leq 0$, the error $r$ cannot change sign.
We treat the case $r < 0$ (i.e.\ $\mu(t) < w$ for all $t$);
the case $r > 0$ is symmetric. Since $r < 0$, the dynamics $\dot{\kappa}_i = -r\,(x_i - \mu)$ give:
\begin{itemize}
\item $x_i > \mu(t)$:\; $\dot{\kappa}_i > 0$ (growing),
\item $x_i < \mu(t)$:\; $\dot{\kappa}_i < 0$ (shrinking, may hit $0$).
\end{itemize}

Since $\mu(t) < w < x_1$, every edge with $x_i \geq w$ satisfies
$x_i > \mu(t)$, so $\dot{\kappa}_i > 0$.
These edges are growing and never reach zero. Moreover, at least one active edge below $w$ exists, or else $\mu \geq w$, contradicting $r < 0$.
So at least one active edge has $x_i < w$, and it is clearly distinct
from $x_1$.

Among active edges, edge~$1$ (with input $x_1$) satisfies $x_1 > \mu$, so
\[
V_{\mathcal{A}} \geq (x_1 - \mu)^2 \geq (x_1 - w)^2 > 0,
\]
since $\mu < w < x_1$.  This bound is uniform in time.
In the case $r > 0$, the symmetric argument shows that $x_N$ always
survives and gives $V_{\mathcal{A}} \geq (x_N - w)^2 > 0$.  In either case,
\[
V_{\min} := \min\!\big((x_1 - w)^2,\, (x_N - w)^2\big) > 0.
\]

By Theorem \ref{Thm:Cond} 
$\|\kappa\|^2 = \sum \kappa_i^2 = 2K$ is conserved.
Since all $\kappa_i \geq 0$,
\[
\bar{\kappa}^2 = \Big(\sum_i \kappa_i\Big)^2
\geq \sum_i \kappa_i^2 = 2K,
\]
and by Cauchy--Schwarz,
\[
\bar{\kappa} \leq \sqrt{N}\,\|\kappa\| = \sqrt{2NK},
\]
so $\bar{\kappa}$ remains bounded above and below for all time.

Combining \eqref{eq:crossbar_loss} with the bounds on $\bar{\kappa}$ and
$V_{\mathcal{A}}$:
\begin{equation}
\label{eq:Phi_dot_crossbar}
\dot{\Phi}
= -\dfrac{2\Phi}{\bar{\kappa}}\, V_{\mathcal{A}}
\leq -\dfrac{2\,V_{\min}}{\bar{\kappa}}\,\Phi
\leq -\dfrac{2\,V_{\min}}{\sqrt{2NK}}\,\Phi
=: -\lambda\,\Phi.
\end{equation}
By Gronwall's inequality,
\[
\Phi(t) \leq \Phi(0)\, e^{-\lambda t}.
\]
The loss converges to zero exponentially, and hence $\mu(t) \to w$.
\end{proof}

\begin{remark}
The EP dynamics for the crossbar are
$\dot{\kappa}_i = -r\,\bar{\kappa}^{-1}(x_i - \mu)$,
which differ from the CL dynamics only by the factor $\bar{\kappa}^{-1}$.
The loss dynamics become
$\dot{\Phi} = -2\Phi\,\bar{\kappa}^{-2}\,V_{\mathcal{A}}$,
and the same argument gives exponential convergence with rate
$\lambda_{\mathrm{EP}} = 2\,V_{\min}/(2NK) = V_{\min}/(NK)$.
\label{Rmk:EP_Crossbar}
\end{remark}

\subsection{Practical implications for training circuits}
\label{sec:practical_implications}

The conservation law reveals that for certain choices of clamping, the overall scale of parameters in the network stays fixed. As we demonstrate here, this has practical importance, since if parameters all grow or shrink indefinitely, the boundaries of allowable values can be hit before a solution is achieved by the circuit. In particular, we examine CL training dynamics for several circuit designs with different nudging choices. We see that, generally speaking, non-conserving nudging can lead parameter trajectories to never converge, or converge only on boundaries of the space where the task may be failed.

\paragraph{Conductance mass drift in linear circuits}

As a general point, training with $\eta \neq 0$ in linear circuits causes the total conductance mass to shift up or down. Recall that this mass is given by
\[
K = \frac{1}{2}\sum_i \kappa_i^2\,.
\]

One can compute the exact pre-$(\eta \to 0)$ learning rule for a linear circuit, by computing \eqref{Learning_dynamics} directly:
\[
\kappa_{i}(t+\alpha) =\kappa_{i}(t) +  \alpha (d_i^\top x_0) (d_i^\top y) - \frac{\alpha \eta}{2}(d_i^\top y)^2\,,
\]
Here $x_0$ and $y$ are the free and clamped states described in \eqref{eq:EP-cirucit} for EP and \eqref{eq:CL-circuit} for CL. Using this update, we can now directly calculate the rate of change of conductance mass,
\begin{equation}\label{eq:mass_drift_discrete}
\begin{aligned}
K(t+\alpha) - K(t) &= \alpha \langle y, \nabla_xG(x;\kappa) \rangle  
\\
&- \alpha\eta G(y;\kappa)+ \sum_i\left(\alpha^2(d_i^\top x)^2(d_i^\top y)^2 +\dfrac{\alpha^2\eta^2}{2}(d_i^\top y)^4\right)
\\
&=- \alpha\eta G(y;\kappa)+ \alpha^2\sum_i\left((d_i^\top x)^2(d_i^\top y)^2 - \eta (d_i^\top x)(d_i^\top y)^3+\dfrac{\eta^2}{2}(d_i^\top y )^4\right)
\end{aligned}
\end{equation}
where we have used the conservation law (Theorem \ref{Thm:Cond}): $\langle y, \nabla_xG(x;\kappa) \rangle =0$.  

As $\alpha \to 0$, we see that 
\begin{equation}\label{eq:mass_drift}
\dot{K} = - \eta G(y;k).
\end{equation}

Recall that $G(x;\kappa) \geq 0$ for all $x$. Hence for $\eta > 0$, $\dot{K} \leq 0$. For a connected circuit, the equality $G(y) = 0$ only holds in the trivial configuration where the voltage $y\equiv 0$ i.e. all nodes have voltage zero in the clamped state, which only occurs when the constraint is satisfied. In practice, we should then see that $K$ decreases (increases) when $\eta > 0$ ($\eta < 0$) until a solution or boundary is reached. In the latter case, the learning rule needs to be modified in an \textit{ad hoc} way, in which case, \eqref{eq:mass_drift} would no longer hold. 

The expression, \eqref{eq:mass_drift},  also provides another perspective on conservation of $K$ when\textit{ symmetric} clamping \eqref{eq:symmetric-clamping} is used. Symmetric clamping compares a positively nudged state with a negatively nudged state instead of comparing the free state to the nudged state. Specifically, $x_0$ is formally replaced by $x_{-\eta}$ in \eqref{eq:continuous-dynamics}.
In this case, the upwards drift on conductances due to negative clamping is offset by the downward drift due to positive clamping. This is demonstrated below in the case of training a voltage divider.

First, we introduce the idea of a batched update, which can be used to train a network to satisfy a set of multiple constraints simultaneously. This technique is commonly used in practical settings, where constraints represent points from some distribution or dataset which is to be learned. 

Consider a number of constraints given by input-output pairs $D = \{(v^{(n)},w^{(n)})\}_{n=1}^B$, and define the expectation empirically:

\[
\langle f(v, w) \rangle = \frac{1}{B}\sum_{n=1}^B f(v^{(n)},w^{(n)})
\]

For any given pair $(v,w)$ of input constraint and desired output, and for general $\eta$, the conductance dynamics are given by

\[
\dot{\kappa}_{i}^{(1)}(v, w) :=  d_i^\top x_0(v) d_i^\top y(v,w) - \frac{\eta}{2}(d_i^\top y(v,w))^2.
\]
To train on the entire constraint set, we take the average over these updates to obtain the learning rule for a batch:

\begin{equation}\label{eq:batched_update}
    \dot{\kappa}_i = \langle\dot{\kappa}_{i}^{(1)}(v, w) \rangle.    
\end{equation}
Note that, if $K$ is conserved by $\dot{\kappa}_{i}^{(1)}(v^{(n)},w^{(n)})$ for all $n$, then $\dot{\kappa}_i$ also conserves $K$. 

\paragraph{Voltage divider}

First we consider a simple voltage divider trained using CL. The voltage divider consists only of three nodes connected in a line: input to output to ground. It is depicted in Figure \ref{fig:voltage_divider}. We can express $x^{(n)} := x_0(v^{(n)})$ and $y^{(n)} := x^{(n)} - (v^{(n)},w^{(n)})$, where we have expressed explicitly the functional dependencies of the free and clamped states on $v$ and $w$. Then
\[
Q^\top x^{(n)} = a(\kappa_1, \kappa_2) v^{(n)} = \frac{\kappa_1}{\kappa_1 + \kappa_2} v^{(n)}.
\]

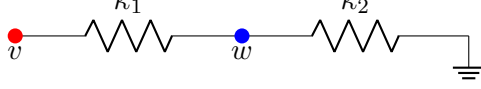
\begin{figure}
\centering
\begin{circuitikz}
  \draw
  (0,0) node[fill=red, circle, inner sep=2pt]{} node[below]{$v$} to[R=$\kappa_1$] (3,0)
  node[fill=blue, circle, inner sep=2pt]{} node[below]{$w$} to[R=$\kappa_2$] (6,0)
  node[ground]{};
\end{circuitikz}
\caption{A simple voltage divider with input $v$, output $w$, and ground}
\label{fig:voltage_divider}
\end{figure}

To write down the dynamics, we only need the second moments. Rescaling $(v, w) \mapsto (\alpha v, \alpha w)$ multiplies the right-hand side of the dynamics by $\alpha^2$, which amounts to a reparameterization of time. Therefore, we set $\langle v^2 \rangle = 1$ without loss of generality and define $c := \langle vw \rangle / s$, so that  

 \begin{equation}\langle v^2 \rangle = 1,\quad  \langle w^2 \rangle = s^2, \quad \langle vw \rangle = sc.
 \label{Moments}
 \end{equation}
Note that $|c| < 1$ by Cauchy--Schwarz. The $\eta = 0$ case gives
\begin{align*}
    \dot{\kappa}_1 &= (a(\kappa)-1)(a(\kappa) - sc) \\
    \dot{\kappa}_2 &= a(\kappa) (a(\kappa) - sc).
\end{align*}
For linear circuits, symmetric clamping gives the exact same dynamics as the $\eta = 0$ case, up to a rescaling of time. Thus, a fortiori, the same is true for a voltage divider. 

At the other extreme, the dynamics for the $\eta = 1$ case are given by
\begin{align*}
    \dot{\kappa}_1 &= \frac{1}{2}(a(\kappa)^2 - s^2) + sc - a(\kappa)\\
    \dot{\kappa}_2 &= \frac{1}{2}(a(\kappa)^2 - s^2).
\end{align*}

In the $\eta = 0$ case, we see that $a(\kappa)=sc$ is the only fixed point of the learning dynamics. Moreover, it is the ``correct'' solution in light of \eqref{Moments}.

At the other extreme, when $\eta=1$, the only fixed points occur on the boundary. In fact, for $\eta \neq 0$, there are no fixed points in the interior. Because the vector field $\kappa^{(1)}_i$ is linear in $\eta$, $\dot{\kappa}$ for any $\eta$ can be obtained as a suitable linear combination of $\dot{\kappa}$ at $\eta=0$ and $\eta=1$. In particular, as noted earlier the flow generated by symmetric clamping as in \eqref{eq:symmetric-clamping} is equivalent to the one obtained with $\eta=0$.

\begin{figure}
    \centering
    \includegraphics[width=1\linewidth]{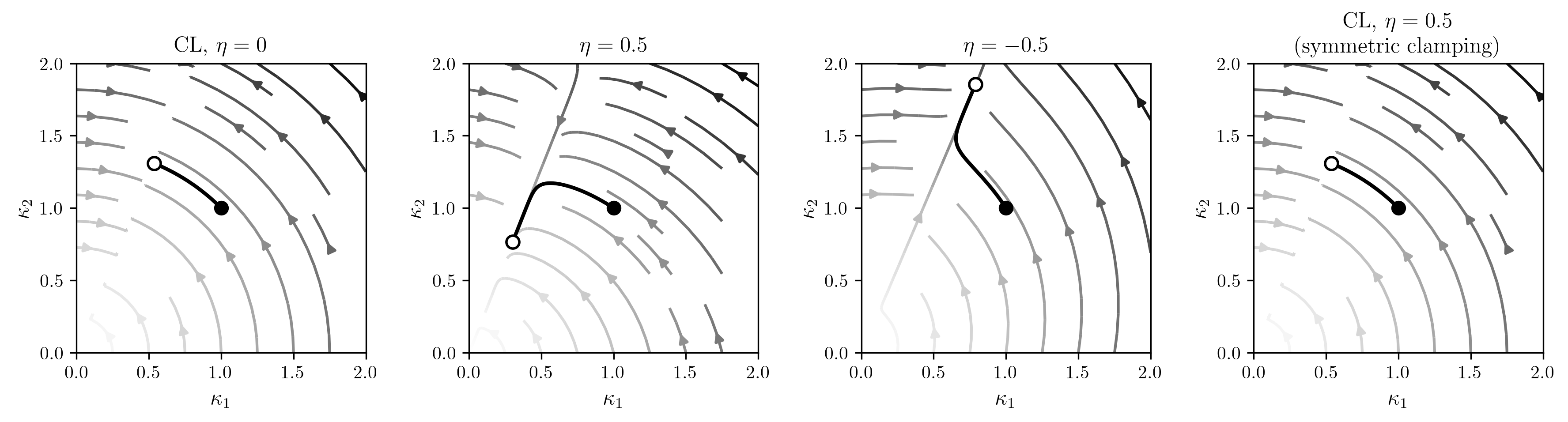}
    \caption{Voltage divider conductance dynamics under Coupled Learning with batched updates and different clamping settings. Vector fields are calculated as a linear combination of the $\eta=0$ and $\eta=1$ cases. Black lines are trajectories computed by simulating the voltage divider circuit. The constraint set is 10 points $(v_n,w_n)$ drawn i.i.d. from a normal distribution with $s \approx 1/3$ and $c \approx 0.9$. Initial and final configurations are given by black and white dots, respectively. The simulation is run for a set amount of time, regardless of whether convergence is achieved.}
    \label{fig:voltage_divider_flows}
\end{figure}

We can understand the lack of convergence in terms of \eqref{eq:mass_drift}. The batched update gives:

\[
\dot{K} = -\eta \langle G(y(v, w);\kappa) \rangle
\]

In the case that there exist two or more constraints which cannot be satisfied (i.e. there does not exist $a \in (0,1)$ s.t. $w_n = a v_n\,\forall n$) we must have that $G(y(v_n, w_n;k)) >0$ for some $n$, and hence $\text{sign} (\dot{K}) = -\text{sign}(\eta)$ while $\kappa_i$ have not hit the boundaries of the space.

While in the $\eta\neq 0$ asymmetric clamping cases the conductances are fated to either hit the origin or grow unbounded, this does not mean an effective solution is not found by the circuit. Notice in Fig \ref{fig:voltage_divider_flows} that for $\eta=0.5$ and $\eta=-0.5$, the trajectory is drawn towards a straight line emanating from the origin. Along this line, the ratio $a(\kappa_1, \kappa_2)$ is constant. Hence, if one stops training before the a parameter boundary is hit, the input-output behavior of the circuit may have nearly converged, even though the conductances are still moving.



\paragraph{Crossbar array}

\begin{figure}
    \centering
    \includegraphics[width=\linewidth]{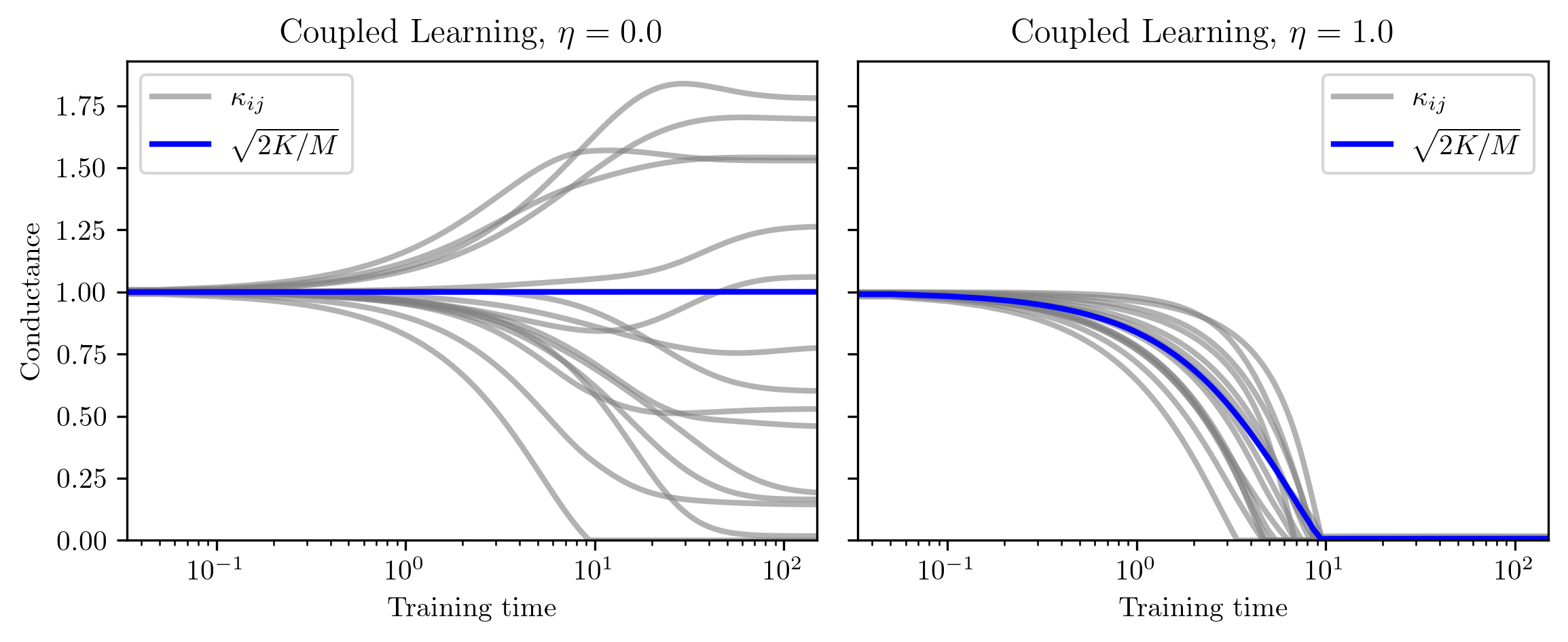}
    \caption{Coupled Learning dynamics of a crossbar array with 5 inputs and 3 outputs. 50 constraint points $(v_n, w_n)$ were sampled iid from a normal distribution and cannot all be satisfied by the circuit. Grey lines show trajectories of individual conductances, while the blue line in each plot shows the RMS value, which is simply related to the total mass conductance mass $K$.}
    \label{fig:crossbar_flows}
\end{figure}

While the crossbar array has the capacity to achieve the input-output map of any other linear circuit, this is not a guarantee that training dynamics will lead it to a solution. In Fig \ref{fig:crossbar_flows}, we see that, as with the voltage divider, using a non-conserving positive nudge causes the conductances shrink over time. Eventually, this causes all of the conductances to hit zero which clearly constitutes a failure of the task. By contrast, using $\eta=0$ (or equivalently, symmetric clamping with any $\eta$) keeps the overall parameter scale the same. Note that this does not guarantee that all conductances remain positive: one edge actually does go to zero at around $t=10$.


\paragraph{Random circuit}

In the voltage divider, learning incommensurate constraints with batched updates at $\eta\neq0$ guarantees that the asymptotic behaviors are $K\to 0$ ($\eta >0$) or $K\to \infty$ ($\eta <0$). Larger circuits with more complicated topologies present a subtler story. Consider a connected linear circuit with hidden nodes (for example, with an ER graph topology) and, without loss of generality, no input-input connections. For a setting with multiple inputs and outputs, we define a set of incommensurate constraints $\{(v_n, w_n)\}$ as one which satisfies:

$$
\min_A \sum_n||w_n - Av_n||^2 > 0
$$

For any circuit, this condition implies that $y(v_n, w_n) \neq 0$ for some $n$. Considering that $y$ is always zero on input nodes, we realize that if any input is connected to any output at which $y(v_n, w_n) \neq 0$ for some $n$, then $\langle G \rangle>0$. Without presenting an exhaustive characterization of the conditions which can lead to this, we simply note that a sufficient condition for $\langle G \rangle=0$ is that all paths from outputs to inputs are removed from the circuit. Hence, for a generic linear circuit, widespread disconnections to outputs can also lead to $\dot{K}=0$, providing an asymptotic behavior which is different from the voltage divider.

\begin{figure}[h]
    \centering
    \includegraphics[width=\linewidth]{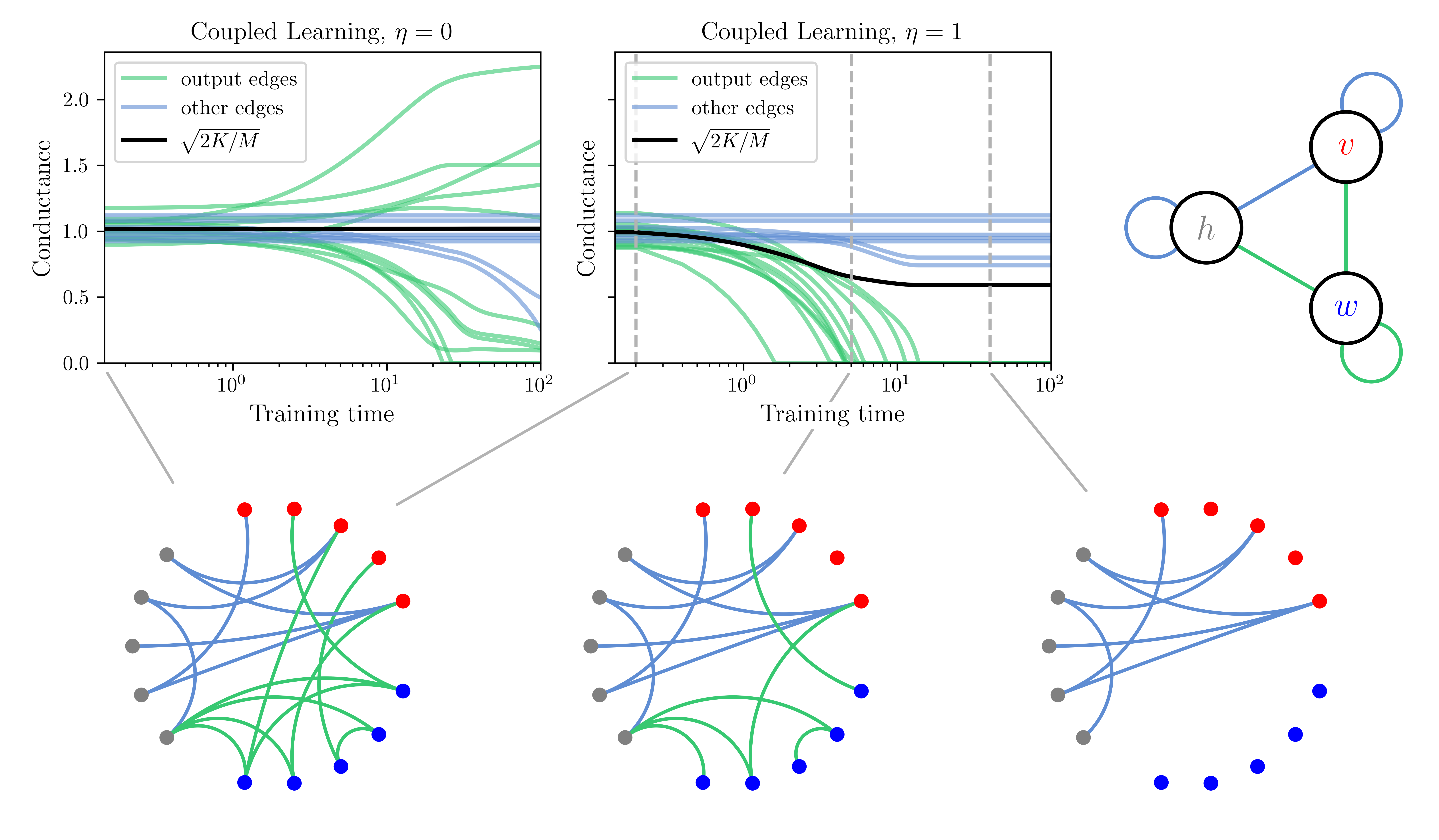}
    \caption{Coupled learning dynamics of a linear circuit with randomly chosen topology. Nodes are colored as inputs: red, outputs: blue, hiddens: grey. Edges are colored by whether they attach to an output. Dataset consists of 40 points $(v_n, w_n)$ drawn iid from a normal distribution with covariance eigenvalues $\lambda_n=n^{-1}$ and random eigenvectors (sampled from uniform measure on orthogonal matrices). In trajectory plots, the black line gives the RMS of conductances in terms of total conductance mass $K$. The networks shown below the trajectories only depict whether an edge is present or absent, not its conductance.}
    \label{fig:er_flows}
\end{figure}

In Fig~\ref{fig:er_flows}, we see this disconnecting asymptotic behavior in an explicit example. Recall that once an edge conductance vanishes, that edge is removed from the circuits. We generated a generic incommensurate constraint set by sampling 40 points from a Gaussian with full-rank covariance, then trained using coupled learning at $\eta=0$ and $\eta=1$. In the $\eta=1$ case, the conductance mass (black line) initially decreases due to $\langle G \rangle > 0$. Then, following a series of disconnection events, paths from inputs to outputs are gradually removed until none remain. In our $\eta=1$ example, all of the edges (green) which disconnect are initially connected to an output node (blue), and all of the other edges (blue) remain after training. This does not need to be the case. In some cases, we see some output edges remain while other edges disappear. Regardless, the asymptotic behavior of this circuit with these constraints appears to be characterized by the removal of edges until no paths from inputs to outputs remain. 

 \section{Discussion}                                                       
  We have shown that physical learning algorithms in the small-nudge regime   conserve a mass-like quantity $K = \tfrac{1}{2}\sum_i f_i(k_i)$ whenever the energy separates as $G(x;k) = \sum_i g_i(k_i)\widetilde{G}_i(x)$ (Lemma~\ref{Lemma:energy_conservation}). For linear circuits this specializes to the
  conserved squared-norm $\tfrac{1}{2}\sum_e \kappa_e^2$ (or $\tfrac{1}{2}\sum_e \rho_e^2$ when training resistances), and we have given three independent proofs of the conservation law: a direct calculation for CL, a gradient-flow argument for EP, and a unified KKT-based proof that handles   both. In linear circuits, symmetric clamping reproduces the $\eta \to 0$   dynamics at any finite $\eta$, so conservation is achievable in hardware without taking small perturbations. We used this conservation law to prove exponential convergence of CL and EP on crossbar arrays (Theorem~\ref{thm:crossbar}) and convergence-or-edge-death in the single-input/single-output setting (Theorem~\ref{Lemma:converge}). We also saw that non-conserving nudging causes the overall scale of the parameters to drift, which can cause training to fail by hitting parameter boundaries before a solution is reached.                                                                                        
  Several questions remain open. First, Remark~\ref{rmk:overdetermined} notes that  allowing $K \colon \mathbb{R}^M \to \mathbb{R}$ to be a general function rather than a sum yields an overdetermined PDE characterizing conserved quantities. We do not know the broadest class of circuits and energies for which nontrivial solutions exist, though preliminary calculations suggest that specific topologies (e.g.\ series-parallel networks) admit higher-order
  polynomial conservations beyond the squared norm. Second,                  Theorem~\ref{Lemma:converge} leaves open the possibility that conductances vanish in such a way that the input becomes disconnected from the output;  Remark~\ref{rmk:edge_death} rules this out in the single-bridge case, but the general topology is harder, since multiple edges can vanish simultaneously and the voltage structure can degenerate in more complicated ways. Third, the   reflecting rule of Section~\ref{Cons of C} requires $\dot{\kappa}_e$ to                    be defined for edges with $\kappa_e = 0$; we cannot at present rule out      pathological circuits where this rule is itself ill-defined. Finally, our    analysis of practical implications was carried out for linear circuits, where symmetric clamping conserves $K$ exactly. For nonlinear circuits symmetric clamping conserves $K$ only up to $O(\eta^2)$, and whether this is enough to prevent the drift pathologies we observed in the linear case is a question for future work. 

\section*{Funding}

YM and AGK were partially supported by NSF through the University of Pennsylvania Materials Research Science and Engineering Center (MRSEC) (DMR-2309043).

\section*{Declaration of competing interest}

The authors declare that they have no known competing financial interests or personal relationships that could have appeared to influence the work reported in this paper.

\section*{Data availability}

Code used to generate the figures in this paper is available from the authors on reasonable request.

\section*{Declaration of generative AI in the writing process}

During the preparation of this work, the authors used Claude (Anthropic) for copy-editing assistance. After using this tool, the authors reviewed and edited the content as needed and take full responsibility for the content of the publication.

\printbibliography

\end{document}